\documentclass{article}

\usepackage{amsthm}
\usepackage{amsmath}
\usepackage{amssymb}
\usepackage{bm}
\usepackage{mathrsfs}
\usepackage[dvips]{graphicx}

\usepackage{algorithmic}
\usepackage{algorithm}

\usepackage{color}

\usepackage{url}

\usepackage{supertabular}

\newtheorem{theorem}{Theorem}
\newtheorem{thm}{Theorem}

\newtheorem{lemma}[thm]{Lemma}

\theoremstyle{definition}

\theoremstyle{remark}

\numberwithin{thm}{section}

\DeclareMathAlphabet{\mathsfsl}{OT1}{cmss}{m}{sl}

\renewcommand{\phi}{\varphi}

\newcommand{\argmin}{\operatorname*{arg\; min}}

\newcommand{\diag}{\operatorname{diag}}

\newcommand{\dist}{\operatorname{dist}}

\newcommand{\bx}{\boldsymbol{x}}

\newcommand{\bZ}{\boldsymbol{Z}}

\newcommand{\bX}{\boldsymbol{X}}

\newcommand{\bD}{\boldsymbol{D}}
\newcommand{\sX}{\mathcal{X}}
\newcommand{\sY}{\mathcal{Y}}

\newcommand{\bU}{\boldsymbol{U}}

\newcommand{\bR}{\boldsymbol{R}}
\def\reals{\mathbb{R}}
\def\bx{\boldsymbol{x}}

\def\bS{\boldsymbol{S}}
\def\b0{\mathbf{0}}

\def\bU{\boldsymbol{U}}
\def\bC{\boldsymbol{C}}

\def\bA{\boldsymbol{A}}
\def\bY{\boldsymbol{Y}}
\def\bB{\boldsymbol{B}}
\def\bH{\boldsymbol{H}}

\def\bI{\mathbf{I}}

\def\tr{\mathrm{tr}}

\usepackage{algorithmic}
\usepackage{algorithm}
\newcommand{\di}{{\,\mathrm{d}}} 
\usepackage{multirow}

\begin{document}
\author{Teng Zhang}
%
\title{A Majorization-Minimization Algorithm for the Karcher Mean of Positive Definite Matrices}
%
%
%

\maketitle

\begin{abstract}An algorithm for computing the Karcher mean of $n$  positive definite matrices is proposed, based on the majorization-minimization (MM) principle. The proposed MM algorithm is parameter-free, does not need to choose step sizes, and has a theoretical guarantee of asymptotic linear convergence.\end{abstract}

\maketitle

\section{Introduction}
It is well-known that the geometric mean for a set of positive real numbers $(a_1,a_2,\cdots, a_n)$ is defined by $(a_1a_2\cdots a_n)^{\frac{1}{n}}$. However, this definition can not be naturally generalized to a set of positive definite matrices $\bA_1,\bA_2,\cdots,\bA_n\in\reals^{p\times p}$, since $(\bA_1\bA_2\cdots \bA_n)^{\frac{1}{n}}$ is usually not symmetric, and it is not invariant to permutation, that is, generally $(\bA_1\bA_2\bA_3\cdots \bA_n)^{\frac{1}{n}} = (\bA_2\bA_1\bA_3\cdots \bA_n)^{\frac{1}{n}}$ does not hold.

The Karcher mean~\cite[(6.24)]{bhatia2007positive}~\cite[Section 4]{JeurVV2012} is commonly used as the geometric mean of positive definite matrices, and it is defined by the optimization problem
%
\begin{equation}\label{eq:main}
\hat{\bX}=\argmin_{\bX\in \SS_+(p)}F(\bX),\,\,\,\text{where}\,\,\,F(\bX)=\sum_{i=1}^n \dist^2(\bA_i,\bX),
\end{equation}
where $\SS_+(p)$ represents the set of all $p\times p$ positive definite matrices, and \begin{equation}\dist(\bA,\bX)=\|\log(\bX^{-\frac{1}{2}}\bA\bX^{-\frac{1}{2}})\|_F.\label{eq:distance}\end{equation} Here $\|\bX\|_F$ denotes the Frobenius norm of $\bX$, and $\bX^{-\frac{1}{2}}$ and $\log\bX$ follow the standard definition of matrix functions~\cite{Higham2008}. 

The solution of \eqref{eq:main} is uniquely defined and satisfies a list of ``desirable properties'' for matrix geometric mean in~\cite[Section 1]{Ando2004}. We refer the reader to~\cite[Section 4]{JeurVV2012} for a more detailed discussion on the proof of its uniqueness, existence and other properties.



The optimization problem \eqref{eq:main} has been extensively investigated in the literature. For example, the gradient descent method has been applied in \cite{Ferreira2006,RenAbs2011}.  A linearization of the gradient descent method in the spirit of
the Richardson iteration is proposed in \cite{Bini2013}, and it is proved to converge locally. 
Another natural algorithm is Newton's method, which is considered in
\cite{Groisser2004,Ferreira2013} in the name of ``centroid computation''. A stochastic algorithm and a gradient descent method are proposed for the Riemannian $p$-means in \cite{Nielsen2013}, and when $p=2$ the Riemannian $p$-means is equivalent to the Karcher mean. A very comprehensive survey~\cite{JeurVV2012} presents several algorithms and their variants, including first-order methods such as the steepest descent method, the conjugate descent method, and second-order methods such as the trust region method and the BFGS method.

A common issue of these algorithms is the choice of step sizes in the update formula. While the line search strategy has a convergence guarantee, it is computationally expensive as observed in \cite{Bini2013}. On the other hand, while the strategy of using constant step sizes converges fast, it lacks theoretical guarantee on the convergence to the Karcher mean, unless the initialization is sufficiently close to the solution (see~\cite[Theorem 2.10]{Afsari2013} for the gradient descent method and \cite[Theorem 5.2]{Groisser2004} for Newton's method). The method of gradually decreasing step sizes in~\cite[Algorithm 3]{Fletcher2007} requires an initial step size, but it is unclear how one should choose this parameter such that the algorithm converges to the Karcher mean. A criterion of choosing step sizes is proposed in \cite{Bini2013}, but it only has a theoretical guarantee on local convergence (although it performs well empirically).

The main contribution of this paper is to present and analyze a majorization-minimization (MM) algorithm for solving~\eqref{eq:main}. Compared to previous methods, the MM algorithm is different and based on the majorization-minimization principle. It is parameter-free, does not need to do line search in each iteration, has asymptotic linear convergence to the Karcher mean.

The rest of the paper is organized as follows. Section~\ref{sec:background} describes the properties of the Karcher mean and the framework of MM algorithms. Section~\ref{sec:algorithm} proposes the MM algorithm for the Karcher mean and analyzes its property of convergence. Section~\ref{sec:simulations} compares the proposed MM algorithm with some previous algorithms under various settings.

\section{Background}\label{sec:background}

\subsection{MM algorithms}

Majorization-minimization (MM) is a principle of designing algorithms. While the name ``MM'' is proposed in recent works by Hunter and Lange~\cite{Hunter2000_2,MM_tutorial2004}, the idea has a long history. For example, the MM principle has been used in the analysis of Weiszfeld's algorithm~\cite{Weiszfeld1937} for finding the Euclidean median~\cite[Section 3.1]{Kuhn73}, and in the analysis of  iterative reweighted least square (IRLS) algorithms for sparse recovery and matrix completion~\cite{Daubechies_iterativelyreweighted,Fornasier_low-rankmatrix}. 

The framework of MM algorithms is as follows. To find $\argmin_{x\in\mathcal{A}}f(x)$, an MM algorithm is an iterative procedure given by
\begin{equation}\label{eq:minimization}
x_{k+1}=T(x_k),\,\,\,\,\text{where}\,\, T(y)=\arg\min_{x\in\mathcal{A}} g(x,y),
\end{equation}
and the majorization surrogate function $g(x,y)$ satisfies
\begin{equation}\label{eq:majorization}\text{$g(x,y)\geq f(x)$ and $g(y,y)=f(y)$.} \end{equation} 
We give a general statement on the convergence of MM algorithm in Theorem~\ref{thm:MM_convergence}.

\begin{theorem}\label{thm:MM_convergence}
If both $f(x)$ and $g(x,y)$ are differentiable with respect to $x$, $f$ is bounded from below, and $T$ is continuous, then for any accumulation point of the sequence $\{x_k\}_{k\geq 1}$, if it lies in the interior of $\mathcal{A}$, then it is a stationary point of $f(x)$.
\end{theorem}
\begin{proof}
First of all, $f(x_k)$ is a nonincreasing sequence:
\begin{eqnarray}\label{eq:MM_descendence}
f(T(x_{k}))=f(x_{k+1})\leq g(x_{k+1},x_k)\leq g(x_k,x_k)=f(x_k).
\end{eqnarray}
Since $f$ is bounded from below, $f(x_k)$ converges. Therefore, $\lim_{k\rightarrow\infty}f(T(x_{k}))-f(x_{k})=0$. Applying the continuity of $f$ and $T$, for any converging subsequence of $\{x_k\}$, $\{x_{m_k}\}\rightarrow\hat{x}$, we have
$f(T(\hat{x}))=f(\hat{x})$, and the equality in (\ref{eq:MM_descendence}) holds if $x_k$ and $x_{k+1}$ are replaced by $\hat{x}$ and $T(\hat{x})$. Therefore, the second inequality in (\ref{eq:MM_descendence}) achieves equality, which means that $\hat{x}$ is a minimizer of $g(x, \hat{x})$. Since $g(x,\hat{x})-f(x)$ is minimized at $x=\hat{x}$, we have $f'({x})\big|_{x=\hat{x}}=g'(x,\hat{x})\big|_{x=\hat{x}}=0$.\end{proof}



The most important component of designing an MM algorithm is to find an appropriate surrogate function $g(x,y)$. A common choice of $g(x,y)$ is a square function, i.e., $c_1(y)x^2+c_2(y)x+c(y)$~\cite[Section 3.1]{Kuhn73,Daubechies_iterativelyreweighted,Fornasier_low-rankmatrix}, which gives a simple update formula in \eqref{eq:minimization}. However, in this paper we will use a surrogate function in the form of $\langle\bC_1(\bX'),\bX\rangle+\langle\bC_2(\bX'),\bX^{-1}\rangle+c_0(\bX')$, where $
\langle\bA,\bB\rangle= \sum_{i,j=1}^{p}\bA_{ij}\bB_{ij}=\tr(\bA\bB^T)$.

\subsection{Matrix derivatives}\label{sec:diff}
Since the analysis in this paper involves matrix derivatives, we review its definition and give several examples in this section. For more details on matrix derivatives, we refer the reader to~\cite{Bhatia1997}.

For a function $f: \reals^{p\times p}\rightarrow \reals$, the directional derivative $D f(\bX)(\bH)$ is defined by
\[
D f(\bX)(\bH)=\lim_{t\rightarrow 0} \frac{f(\bX+t\bH)-f(\bX)}{t}.
\]
We say $D f(\bX) =\bY$ if $D f(\bX)(\bH)=\langle \bY, \bH\rangle$.

Next we give some examples that will be used later. Since we work with symmetric matrices throughout the paper, we assume that the matrices $\bX$ and $\bA$ are symmetric in the following examples.

 A simple example is $f(\bX)=\tr(\bA\bX)$, for which we have $D f(\bX)=\bA$. 

 For $f(\bX)=\langle \bX^{-1}, \bA\rangle$, following a well-known result on the derivatives of matrix inverse~\cite{IMM2012-03274},
\begin{equation}D f(\bX) = -\bX^{-1}\bA\bX^{-1}.\label{eq:matrix_deri1}\end{equation}


For $f(\bX)=\|\log \bX\|_F^2$, applying the result in \cite[pg. 218]{bhatia2007positive}, we have $D f(\bX) = 2 \bX^{-1}\log\bX.$

For $f(\bX)=\|\bX\bA\|_F^2$, we have $D f(\bX) = 2 \bA\bX\bA$,
since
\begin{align*}
&\|(\bX+ t\bH)\bA\|_F^2-\|\bX\bA\|_F^2
=\tr\Big((\bX+ t\bH)\bA(\bX+ t\bH)\bA-\bX\bA\bX\bA\Big)
\\ = & 2t\,\tr\Big(\bX\bA\bH\bA\Big) + O(t^2)
=\langle2t\,\bA\bX\bA,\bH\rangle+ O(t^2) .
\end{align*}
\section{MM algorithm for computing the Karcher mean}\label{sec:algorithm}
We first present the majorization-minimization (MM) algorithm for \eqref{eq:main}:
\begin{equation}\label{eq:algorithm}
\bX_{k+1}=T(\bX_k)=f_2(\bX_k)^{\frac{1}{2}}\big(f_2(\bX_k)^{\frac{1}{2}}f_1(\bX_k)f_2(\bX_k)^{\frac{1}{2}}\big)^{-\frac{1}{2}}f_2(\bX_k)^{\frac{1}{2}},
\end{equation}
where
\[
f_1(\bX)=\sum_{i=1}^n\bA_i^{-\frac{1}{2}}g_1(\bA_i^{-\frac{1}{2}}\bX\bA_i^{-\frac{1}{2}})\bA_i^{-\frac{1}{2}}, \,\,\,g_1(x)=(\sqrt{\log x^2+1}+\log x)x^{-1},
\]
and
\[
f_2(\bX)=\sum_{i=1}^n\bA_i^{\frac{1}{2}}g_2(\bA_i^{-\frac{1}{2}}\bX\bA_i^{-\frac{1}{2}})\bA_i^{\frac{1}{2}},\,\,\,\text{ $g_2(x)=(\sqrt{\log x^2+1}-\log x)x$}.
\]

We follow the definition of the matrix functions in \cite{Higham2008}. Especially, for a symmetric matrix $\bX$, the matrix function $g(\bX)$ is defined as follows: Assume that the eigenvalue decomposition is given by $\bX=\bU\diag(\lambda_1,\lambda_2,\cdots,\lambda_p)\bU^T$, then
 \begin{equation}g(\bX)=\bU\diag\big(g(\lambda_1),g(\lambda_2),\cdots, g(\lambda_p)\big)\bU^T.\label{eq:matrixfunction}\end{equation}

We have the following theorem on the convergence of the  proposed MM algorithm:
\begin{thm}\label{thm:main}
The sequence $\{\bX_k\}_{k\geq 1}$ generated by~\eqref{eq:algorithm} converges to the solution of \eqref{eq:main}, and $\{F(\bX_k)\}_{k\geq 1}$ converges linearly asymptotically.
\end{thm}

The proof of Theorem~\ref{thm:main} is based on the following Lemmas. The proof of Theorem~\ref{thm:main} will be given in Section~\ref{sec:proof_main} and the proof of the Lemmas will be given in Section~\ref{sec:lemma}. 

\begin{lemma}\label{lemma:major}
There exists $c_0(\bX'): \reals^{p\times p}\rightarrow \mathbb{R}$ such that \begin{equation}\label{eq:define_G}G(\bX,\bX')= \langle f_1(\bX'),\bX \rangle + \langle f_2(\bX'),\bX^{-1} \rangle +c_0(\bX')\end{equation}
satisfies $G(\bX',\bX')=F(\bX')$ and $G(\bX,\bX')\geq F(\bX)$.
\end{lemma}

\begin{lemma}\label{lemma:minimization}For any positive definite matrices $\bC_1,\bC_2\in\reals^{p\times p}$, the minimizer of  $\langle\bC_1,\bX\rangle+\langle\bC_2,\bX^{-1}\rangle$ is
$\bC_2^{\frac{1}{2}}(\bC_2^{\frac{1}{2}}\bC_1\bC_2^{\frac{1}{2}})^{-\frac{1}{2}}\bC_2^{\frac{1}{2}}.$
\end{lemma}

\begin{lemma}\label{lemma:derivative}
For a twice differentiable function $f(x): \reals\rightarrow\reals$, we have:\\(a)If $f''(x)\geq \mu>0$ for all $x\in\reals$, then $f(x_0)-\min_{x\in\reals} f(x) \leq \frac{f'(x_0)^2}{2\mu}$. \\(b)If $f''(x)\leq L$ for all $x\in\reals$, then $f(x_0)-\min_{x\in\reals} f(x)\geq \frac{f'(x_0)^2}{2L}$.
\end{lemma}

The proof of Theorem~\ref{thm:main} also depends on $F(\bX)$ and $G(\bX,\bX')$ when evaluated on geodesic lines in $\SS_+(p)$. 
By~\cite[Theorem 6.1.6]{bhatia2007positive}, the geodesic line connecting $\bA$ and $\bB$ is parameterized by
\begin{equation}\label{eq:geodesic}
L(t)=\bA^{\frac{1}{2}}(\bA^{-\frac{1}{2}}\bB\bA^{-\frac{1}{2}})^{t}\bA^{\frac{1}{2}},\,\,\,\,\, t\in [0,1].
\end{equation}
In particular, this is an arc length parameterization when $\dist(\bA,\bB)=1$, i.e., $\|\log (\bA^{-\frac{1}{2}}\bB\bA^{-\frac{1}{2}})\|_F=1$. We will always use arc length parameterizations throughout the paper. The results on $F(L(t))$ and $G(L(t),\bX')$ are summarized as follows.
\begin{lemma}\label{lemma:derivative1}
(a)For any geodesic line $L\in\SS_+(p)$ and $t\in\reals$, $F''(L(t))\geq 2n$.\\(b) There exists $C>0$ such that for any geodesic line $L$ and $\bX'\in\SS_+(p)$ that satisfy $\dist(L(0),\hat{\bX})\leq 1$ and $\dist(\bX',\hat{\bX})\leq 1$, $\frac{\di^2}{\di t^2}G(L(t),\bX')\Big|_{t=0}<C$.
\end{lemma}

\subsection{Proof of Theorem~\ref{thm:main}}\label{sec:proof_main}

 By Lemmas~\ref{lemma:major} and~\ref{lemma:minimization}, the iterative procedure satisfies the definition of MM algorithm in \eqref{eq:minimization} and \eqref{eq:majorization}. Because $F$ is strictly  geodesically convex~\cite{JeurVV2012}, the minimizer $\hat{\bX}$ is the unique stationary point of $F(\bX)$. Applying Theorem~\ref{thm:MM_convergence} (which is applicable since both $F$ and $G$ are differentiable), 
\begin{equation}\label{eq:cc}\text{any converging subsequence of $\{\bX_k\}_{k\geq 1}$ converges to $\hat{\bX}$.}\end{equation}

By the monotonicity of MM algorithms in~\eqref{eq:MM_descendence}, $F(\bX_k)$ is nonincreasing and the sequence $\{\bX_k\}_{k\geq 1}$ is contained in the level set $\sX_0=\{\bX: F(\bX)\leq F(\bX_1)\}$. Let $\sY_t=\{\bx: \|\bX\|\leq t, \|\bX^{-1}\|\leq t\}$, then following the proof of~\cite[Theorem 2.4]{Bini2014}, $t$ can be chosen sufficiently large such that $F(\bX)>F(\bX_1)$ for all positive definite matrices $\bX$ such that $\bX\notin\sY_t$. As a result, $\sX_0\subseteq\sY_t$, and the sequence $\{\bX_k\}_{k\geq 1}$ is contained in $\sY_0$.

Since $\sY_0$ is a compact set (it is closed and bounded), every subsequence of $\{\bX_k\}_{k\geq 1}$ has a converging sub-subsequence, which converges to $\hat{\bX}$ according to \eqref{eq:cc}. Applying~\cite[Excercise 2.11.20]{thomson2008elementary}, $\{\bX_k\}_{k\geq 1}$ converges to $\hat{\bX}$.

Next we will show that the proposed MM algorithm converges linearly asymptotically.

Parametrize the geodesic line connecting $\bX_k$ and $\hat{\bX}$ by $L(t)$ such that $L(0)=\bX_k$, then by Lemma~\ref{lemma:derivative1}(a),  $F(L(t))''\geq 2n$ for all $t$. Applying Lemma~\ref{lemma:derivative}(a) to $F(L(t))$,  we have
\[
F(\bX_k)-F(\hat{\bX})=F(L(0))-\min_t F(L(t))\leq \frac{\di}{\di t}F(L(t))\Big|_{t=0}^2\Big/4n.
\]

Due to the convergence of $\{\bX_k\}_{k\geq 1}$, there exists $K$ such that for any $k>K$, $\dist(\bX_k,\hat{\bX})<1$. Applying Lemma~\ref{lemma:derivative1}(b) and Lemma~\ref{lemma:derivative}(b) to $G(L(0),\bX_k)$, for $k>K$,
\begin{align*}
F(\bX_k)-F(\bX_{k+1})\geq & G(\bX_k,\bX_k)-\min_{\bX\in \SS_+(p)}G(\bX,\bX_k)
\geq G(L(0),\bX_k)-\min_{t\in\reals }G(L(t),\bX_k)
\\\geq & \frac{\di}{\di t}G(L(t),\bX_k)\Big|_{t=0}^2\Big/2C,
\end{align*}
where the first inequality applies \eqref{eq:MM_descendence}, and $C$ is the positive constant obtained in Lemma 3.4 (with $\bX_k$ replacing $\bX'$).

By Lemma~\ref{lemma:major}, $G(L(t),\bX_k)-F(L(t))$ is minimized at $t=0$, which implies $\frac{\di}{\di t}F(L(t))\Big|_{t=0}=\frac{\di}{\di t}G(L(t),\bX_k)\Big|_{t=0}$. Therefore, $F(\bX_k)-F(\bX_{k+1})\geq \frac{2n}{C}(F(\bX_k)-F(\hat{\bX}))$ and \begin{equation}\label{eq:convergencerate}F(\bX_{k+1})-F(\hat{\bX})\leq (1-\frac{2n}{C})(F(\bX_k)-F(\hat{\bX})).\end{equation} Since $\hat{\bX}$ is the minimizer of $F(\bX)$, we have $F(\bX_{k+1})-F(\hat{\bX})\geq 0$ and $F(\bX_{k})-F(\hat{\bX})\geq 0$. Combining it with \eqref{eq:convergencerate}, we have that $1-\frac{2n}{C}\geq 0$. Since both $n$ and $C$ are nonnegative, $0\leq 1-\frac{2n}{C}<1$. Then \eqref{eq:convergencerate} implies that $F(\bX_k)$ converges linearly to $F(\hat{\bX})$  for $k>K$.

\subsection{Proof of  Lemmas}\label{sec:lemma}
\subsubsection{Proof of Lemma~\ref{lemma:major}}\label{sec:lemma1}

We start with the following auxiliary lemma and its proof:
\begin{lemma}\label{lemma:major1}
$\bX'$ is the unique minimizer of \begin{align}\label{eq:major1}
g_{\bX'}(\bX)=&\big\langle g_1(\bX'),\bX\big\rangle +
\big\langle g_2(\bX'),\bX^{-1}\big\rangle
- \|\log\bX\|_F^2\nonumber
\end{align}
over the set $\SS_+(p)$.
\end{lemma}
\begin{proof}
The proof can be divided into two steps. In the first step, we will show that $\bX'$ is the unique stationary point of $g_{\bX'}(\bX)$. In the second step, we will show that $\bX'$ is the unique minimizer of $g_{\bX'}(\bX)$.

We start with the first step. Applying the matrix derivatives in Section~\ref{sec:diff}, $g_{\bX'}(\bX)$ is differentiable and the derivative with respect to $\bX$ is
\[
g_1(\bX')-\bX^{-1}g_2(\bX')\bX^{-1}-2\bX^{-1}\log\bX.
\]
Let $\bZ=g_2(\bX')$ and apply $g_1(\bX')=g_2(\bX')^{-1}$ (which follows from $g_1(x)=g_2(x)^{-1}$ and the definition of matrix function), the equation for the stationary point of $g_{\bX'}(\bX)$ is given by
\begin{equation}\label{eq:stationary}
\bZ^{-1}-\bX^{-1}\bZ\bX^{-1}-2\bX^{-1}\log\bX=0.
\end{equation}

Apply the matrix derivatives in Section~\ref{sec:diff}, the LHS of \eqref{eq:stationary} is the derivative of
\[
h(\bZ)=\log\det(\bZ)-\frac{1}{2}\|\bZ\bX^{-1}\|_F^2-\tr(2\bX^{-1}\log\bX\bZ)
\]
with respect to $\bZ$. $h(\bZ)$ is concave with respect to $\bZ$: $\log\det(\bZ)$ is concave with respect to $\bZ$~\cite[Section 3.1.5]{boyd2004convex}, $\tr(2\bX^{-1}\log\bX\bZ)$ is a linear function with respect to $\bZ$, and $\|\bZ\bX^{-1}\|_F^2$ is convex respect to $\bZ$. Indeed, we can prove the convexity of  $\|\bZ\bX^{-1}\|_F^2$ as follows: let $\bZ_1$ and $\bZ_2$ be two arbitrary matrices in $\reals^{p\times p}$, then
\begin{align*}
&\|\bZ_1\bX^{-1}\|_F^2+\|\bZ_2\bX^{-1}\|_F^2-2\Big\|\frac{\bZ_1+\bZ_2}{2}\bX^{-1}\Big\|_F^2\\
=&\langle
\bZ_1^2,\bX^{-2}\rangle+\langle
\bZ_2^2,\bX^{-2}\rangle-2\Big\langle
\Big(\frac{\bZ_1+\bZ_2}{2}\Big)^2,\bX^{-2}\Big\rangle
=2\Big\langle
\Big(\frac{\bZ_1-\bZ_2}{2}\Big)^2,\bX^{-2}\Big\rangle
\\=&2\Big\|\frac{\bZ_1-\bZ_2}{2}\bX^{-1}\Big\|_F^2\geq 0.
\end{align*}
This established the ``midpoint convexity'' of $\|\bZ\bX^{-1}\|_F^2$, which is defined as follows: for any function $f(x)$, midpoint convexity means $f(x_1)+f(x_2)\geq 2f(\frac{x_1+x_2}{2})$ for all $x_1$ and $x_2$.

Following the proof of~\cite[Theorem 1.1.4]{niculescu2006convex}, for any continuous function, the midpoint convexity is equivalent to convexity. since $\|\bZ\bX^{-1}\|_F^2$ is a continuous function and has midpoint convexity, it is a convex function.

Applying the concavity of $h(\bZ)$, its stationary point is unique. That is, when $\bX$ is given, there is a unique $\bZ$ such that \eqref{eq:stationary} holds. By calculation, it is easy to verify that this unique solution is $\bZ=g_2(\bX)$. Therefore, any $(\bX,\bZ)$ satisfying \eqref{eq:stationary} satisfies $\bZ=g_2(\bX)$. Next we will prove $\bX=\bX'$ by combining it with $\bZ=g_2(\bX')$.

Since $g_2(x)'=\Big(1-\frac{1}{\sqrt{\log^2 x+1}}\Big)\Big(\sqrt{\log^2 x+1}-\log x\Big)\geq 0$
and $g_2(x)'=0$ holds only when $x=1$, $g_2(x)$ is monotonically increasing and $g_2^{-1}(x)$ is uniquely defined. Denote the $p$ eigenvalues of a matrix $\bX\in\SS_+(p)$ by $\lambda_1(\bX)\geq \lambda_2(\bX)\geq \cdots \geq \lambda_p(\bX)$, then by~\cite[page 526]{meyer2000matrix}, $g_2(\bX)=\bZ$ implies that $\bX$ and $\bZ$ have the same set of eigenvectors and $\lambda_i(\bZ)=g_2(\lambda_i(\bX))$ for all $1\leq i\leq p$. Since $g_2^{-1}$ is uniquely defined, the eigenvalues and the eigenvectors of $\bX$ are both uniquely defined. Therefore, the solution to $\bZ=g_2(\bX)$ is uniquely given by $\bX=\bX'$, and the unique solution to \eqref{eq:stationary} is given by $\bX=\bX'$. This concludes the first step of the proof.

In the second step of the proof, we will show that $g_{\bX'}(\bX)$ goes to $\infty$ when $\lambda_1(\bX)\rightarrow\infty$ or $\lambda_p(\bX)\rightarrow 0$. Indeed, let $c_1=\lambda_p(g_1(\bX'))$ and $c_2=\lambda_p(g_2(\bX'))$, then it can be proved by combining \[g_{\bX'}(\bX)\geq c_1\tr(\bX)+c_2\tr(\bX^{-1})-\|\log(\bX)\|_F^2=\sum_{i=1}^p(c_1\lambda_i(\bX)-c_2/\lambda_i(\bx)-\log^2\lambda_i(\bX)),\] and the fact that for any $c_1, c_2>0$, $c_1x+c_2x^{-1}-\log^2 x\rightarrow \infty$ when $x\rightarrow 0$ or $x\rightarrow\infty$.

Since $g_{\bX'}(\bX)$ is a continuous function, there exists $M,m>0$ such that the minimizer of $g_{\bX'}(\bX)$ is in the set $\{\bX\in\SS_+(p): \lambda_1(\bX)\leq M, \lambda_p(\bX)\geq m\}$. This set is compact because $\lambda_1(\bX)$ and $\lambda_p(\bX)$ are continuous functions with respect to $\bX$ (which can be proved by applying the Bauer-Fike Theorem~\cite{Bauer1960}). Recall that this compact set has only one stationary point, this stationary point is also the unique minimizer of $g_{\bX'}(\bX)$. That is, $\bX=\bX'$ is the unique minimizer of $g_{\bX'}(\bX)$.\end{proof}

\begin{proof}[Proof of Lemma~\ref{lemma:major}]
Applying Lemma~\ref{lemma:major1}, we have
 \begin{equation}\label{eq:major6}
 g_{\bX'}(\bX)-g_{\bX'}(\bX')\geq 0,\,\,\,g_{\bX'}(\bX')-g_{\bX'}(\bX')=0.
 \end{equation}Replace $\bX$, $\bX'$ in \eqref{eq:major6} by $\bA_i^{-\frac{1}{2}}\bX\bA_i^{-\frac{1}{2}}$, $\bA_i^{-\frac{1}{2}}\bX'\bA_i^{-\frac{1}{2}}$, and summing it over $1\leq i\leq n$, we proved Lemma~\ref{lemma:major} with $c_0(\bX')$ in \eqref{eq:define_G} defined by
 \begin{equation}\label{eq:major7}
 c_0(\bX')=-\sum_{i=1}^n g_{\bA_i^{-\frac{1}{2}}\bX'\bA_i^{-\frac{1}{2}}}(\bA_i^{-\frac{1}{2}}\bX'\bA_i^{-\frac{1}{2}}).
 \end{equation}
\end{proof}


%

\subsubsection{Proof of Lemma~\ref{lemma:minimization}}\label{sec:lemma2}

Since $\bX^{-1}$ is operator convex~\cite[Theorem 2.6]{Carlen2010}, i.e.,
$(\bX+\bY)^{-1}+(\bX-\bY)^{-1}- 2\bX^{-1}$
 is positive definite, $\langle\bC_2,\bX^{-1}\rangle$ is midpoint convex:
\begin{align*}
&\langle\bC_2,(\bX+\bY)^{-1}\rangle + \langle\bC_2,(\bX-\bY)^{-1}\rangle - 2\,\langle\bC_2,\bX^{-1}\rangle
\\=& \langle\bC_2,(\bX+\bY)^{-1}+(\bX-\bY)^{-1} -2 \bX^{-1} \rangle\geq 0,
\end{align*}
where the last inequality applies the property that for any two positive semidefinite matrices $\bA,\bB$,  $\langle\bA,\bB\rangle\geq 0$.
Following the proof of~\cite[Theorem 1.1.4]{niculescu2006convex}, $\langle\bC_2,\bX^{-1}\rangle$ is convex.

Since $\langle\bC_1,\bX\rangle$ is a linear function about $\bX$, $\langle\bC_1,\bX\rangle+\langle\bC_2,\bX^{-1}\rangle$ is convex and the unique minimizer is the root of its derivative, i.e., the solution to \begin{equation}\label{eq:matrix_deri0}
\bC_1-\bX^{-1}\bC_2\bX^{-1}=0.
\end{equation}

Lemma~\ref{lemma:minimization} is then proved by verifying that $\bX=\bC_2^{\frac{1}{2}}(\bC_2^{\frac{1}{2}}\bC_1\bC_2^{\frac{1}{2}})^{-\frac{1}{2}}\bC_2^{\frac{1}{2}}$ satisfies \eqref{eq:matrix_deri0}.
%

\subsubsection{Proof of Lemma~\ref{lemma:derivative}}
(a) When $f''(x)\geq \mu>0$, $f$ is a strongly convex function. Assume $x^*=\argmin_{x\in\reals} f(x)$, applying~\cite[Theorem 2.1.10]{nesterov2004introductory} with $f'(x^*)=0$, we have
\[
f(x_0)-f(x^*)\leq\langle f'(x^*),x_0-x^*\rangle +\frac{1}{2\mu}\|f'(x_0)-f'(x^*)\|^2=\frac{1}{2\mu}\|f'(x_0))\|^2.
\]

(b) When $f''(x)\leq L$, $f(x)$ satisfies~\cite[equation (2.1.6)]{nesterov2004introductory}. Then applying~\cite[(2.1.7)]{nesterov2004introductory} to $x_0$ and $x^*$ with $f'(x^*)=0$, we have
\[
f(x_0)-f(x^*)\geq \langle f'(x^*),x_0-x^*\rangle +\frac{1}{2L}\|f'(x_0)-f'(x^*)\|^2=\frac{1}{2L}f'(x_0)^2.
\]
\subsubsection{Proof of Lemma~\ref{lemma:derivative1}}
(a) 
Applying the semiparallelogram law \cite[(6.16)]{bhatia2007positive}, for any $1\leq i\leq n$,
\begin{align*}
&\dist^2(\bA_i,L(t+\epsilon))-2\dist^2(\bA_i,L(t))+\dist^2(\bA_i,L(t-\epsilon))
\\\geq& \frac{1}{2}\dist^2(L(t+\epsilon),L(t-\epsilon)) = 2\epsilon^2.
\end{align*}
The lemma can be proved by combining it with $F(\bX)=\sum_{i=1}^n\dist^2(\bA_i,\bX)$ and
\[F''(L(t))=\lim_{\epsilon\rightarrow 0}\frac{F(L(t+\epsilon))-2F(L(t))+F(L(t-\epsilon))}{\epsilon^2}\]
(b) Parameterize all geodesic lines by $L(t)=L_{\bX,\xi}(t)=\bX^{\frac{1}{2}}\exp(t\xi)\bX^{\frac{1}{2}}$, where $\bX\in\SS_+(p)$ and $\xi\in\reals^{p\times p}$ satisfies $\|\xi\|_F=1$ (so that $L(t)$ is an arc length parameterization), we will show that $\frac{\di^2}{\di t^2}G(L_{\bX,\xi}(t),\bX')\Big|_{t=0}$ is a continuous function with respect to $\bX, \xi$ and $\bX'$ by showing that this property holds for both $G_1(L_{\bX,\xi}(t),\bX')=\langle f_1(\bX'), L_{\bX,\xi}(t) \rangle$ and $G_2(L_{\bX,\xi}(t),\bX')=\langle f_2(\bX'), L_{\bX,\xi}(t) \rangle$.

Let $G_1(L_{\bX,\xi}(t),\bX')=\langle f_1(\bX'), L_{\bX,\xi}(t) \rangle$, then by definition,
\begin{align*}
\frac{\di^2}{\di t^2} G_1(L_{\bX,\xi}(t),\bX')\Big|_{t=0}=&\lim_{t\rightarrow 0}\frac{\langle f_1(\bX'), L_{\bX,\xi}(t)-2L_{\bX,\xi}(0)+L_{\bX,\xi}(-t) \rangle}{t^2}
\\=&\lim_{t\rightarrow 0}\frac{\langle f_1(\bX'), \bX^{\frac{1}{2}}(\exp(t\xi)-2\bI+\exp(-t\xi))\bX^{\frac{1}{2}}\rangle}{t^2}.
\end{align*}
Applying the Taylor expansion $\exp(t\xi)=\bI+t\xi+\frac{t^2}{2}\xi^2+o(t^2)$, the derivative is ${\langle f_1(\bX'), \bX^{\frac{1}{2}}\xi^2\bX^{\frac{1}{2}}\rangle},$
which is well-defined and continuous with respect to $\bX, \xi$ and $\bX'$.

Similarly we can prove the same property for $G_2(L_{\bX,\xi}(t),\bX')$ and therefore $\frac{\di^2}{\di t^2}G(L_{\bX,\xi}(t),\bX')\Big|_{t=0}$ is a continuous function with respect to $\bX, \xi$ and $\bX'$.

Recall that $\hat{\bX}$ defined in \eqref{eq:main} is the minimizer of $F(\bX)$, the set of all parameters $(\bX,\bX',\xi)$ that satisfy the assumptions in Lemma~\ref{lemma:derivative1}(b) is given by $\{(\bX,\bX',\xi)$:  $\dist(\bX,\hat{\bX})\leq 1, \dist(\bX',\hat{\bX})\leq 1, \|\xi\|_F=1\}$, which is a compact set. Combining the compactness with the continuity of $\frac{\di}{\di t}G(L_{\bX,\xi}(t),\bX')\Big|_{t=0}$ with respect to  $\{\bX,\bX',\xi\}$, part (b) is proved.

\subsection{Discussion of the majorization function}
 First, we explain why Lemma~\ref{lemma:major1} is important for the choice of the majorization function of $F(\bX)$: Assume that the majorization function is in the form of
 \begin{equation}\label{eq:majorr}\langle\bC_1,\bX\rangle+\langle\bC_2,\bX^{-1}\rangle+c_0,\end{equation} then a natural idea is to find a majorization function in the form of \eqref{eq:majorr} for each component of $F(\bX)$, i.e., $\|\log(\bA^{-\frac{1}{2}}\bX\bA^{-\frac{1}{2}})\|_F^2$. Let $\bY=\bA^{-\frac{1}{2}}\bX\bA^{-\frac{1}{2}}$, then it is equivalent to find a majorizing function of $\|\log \bY\|_F^2$ in the form of 
 \begin{equation}\langle\bA^{\frac{1}{2}}\bC_1\bA^{\frac{1}{2}},\bY\rangle+\langle\bA^{-\frac{1}{2}}\bC_2\bA^{-\frac{1}{2}},\bY^{-1}\rangle+c_0.\end{equation}
If Lemma~\ref{lemma:major1} holds, then it is clear that $\bA^{\frac{1}{2}}\bC_1\bA^{\frac{1}{2}}=g_1(\bY')$ and $\bA^{-\frac{1}{2}}\bC_2\bA^{-\frac{1}{2}}=g_2(\bY')$ would suffice.

Therefore, the problem has been reduced to finding $g_1(\bX)$ and $g_2(\bX)$ such that Lemma~\ref{lemma:major1} holds. Actually, $g_1$ and $g_2$ in Lemma~\ref{lemma:major1} is motivated by the analysis of the case $p=1$.

When $p=1$, the goal is to choose  $g_1(x)$ and $g_2(x)$ such that $x'$ is the unique minimizer of $g_1(x')x + g_2(x')/x - \log^2x.$ Let $z=\log x$ and $z'=\log x'$, then it is equivalent to find $g_1(z)$ and $g_2(z)$ such that $z'$ is the unique minimizer of
\[
g_0(z) = g_1(z')e^z+g_2(z')e^{-z}-z^2.
\]

To achieve the goal, it suffices to have $g_0'(z')=0$ and $g_0''(z)\geq 0$ for all $z\in\reals$, that is,  \begin{equation}\label{eq:1d_2}
g_1(z')e^{z'}-g_2(z')e^{-z'}-2z'=0,\,\,\,\,\, g_1(z')e^{z}+g_2(z')e^{-z}\geq 2.
\end{equation}
By the Cauchy-Schwartz inequality, the second equation in \eqref{eq:1d_2} is satisfied when $g_1(z')g_2(z') = 1.$ Combining it with the first equation in \eqref{eq:1d_2}, we have
\[
g_1(z')=e^{-z'}(\sqrt{z'^2+1}+z'),\,\,\,g_2(z')=e^{z'}(\sqrt{z'^2+1}-z').
\]
Plug in $z'=\log x'$, we obtain $g_1$ and $g_2$ in Lemma~\ref{lemma:major1}.

\subsection{Computational Cost}\label{sec:cost}
The computation cost of the MM algorithm mainly comes from the evaluation of matrix functions, including square root, logarithm, inverse square root, $g_1(\bX)$ and $g_2(\bX)$.

The standard way of calculating matrix functions is through Schur decomposition~\cite{Higham2008}. For positive definite matrices, Schur decomposition is equivalent to eigenvalue decomposition and the matrix function is given by the matrix multiplication in \eqref{eq:matrixfunction}. Therefore, eigenvalue decomposition is the main computational cost in each step of the MM algorithm.

Now we will calculate the number of eigenvalue decompositions in the MM algorithm. Assuming that for all $1\leq i\leq n$, $\bA_i^{-\frac{1}{2}}$ and $\bA_i^{\frac{1}{2}}$ are computed in advance, then in each iteration we need to calculate matrix functions for $\bA_i^{-\frac{1}{2}}\bX_k\bA_i^{-\frac{1}{2}}$ ($g_1$, $g_2$), $f_2(\bX_k)$ (square root) and $f_2(\bX_k)^{\frac{1}{2}}f_1(\bX_k)f_2(\bX_k)^{\frac{1}{2}}$ (inverse square root). Therefore, the algorithm requires $n+2$ eigenvalue decompositions in each iteration. 

\section{Simulations}\label{sec:simulations}
There are many other algorithms for computing the Karcher mean of positive definite matrices, but the gradient descent and its variants are more commonly used. Indeed, \cite{JeurVV2012} gave a extensive survey on various algorithms such as the steepest descent method (SD), the conjugate gradient method (CG), Riemannian BFGS method (RBFGS), and the trust region method (TR) with the Armijo line search technique. It is shown that while CG has a similar performance as SD, the second order methods, including RBFGS and TR, are outperformed by SD and CG when the size of matrices increases.

In this section we compare the MM algorithm with a linearized gradient descent algorithm with a Richardson-like iteration~\cite{Bini2013}: let the Cholesky decomposition of $\bX_k$ be $\bX_k=\bR_k^T\bR_k$, then
\begin{align}\nonumber
&\bX_{k+1}=\bX_k-\nu_k \bR_k^T\sum_{i=1}^n \log(\bR_k^{-1\,T}\bA_i\bR_k^{-1})\bR_k,\label{eq:toolbox}
\end{align}
with $\nu_k$ is chosen to be the optimal value  \cite[(9)]{Bini2013}.
 We use the code available at http://bezout.dm.unipi.it/software/mmtoolbox/, and we referred the algorithm as ``Toolbox'' in the simulations. We also compare the MM algorithm with the gradient descent method (GD)~\cite{Xavier06} with a line search procedure, which is described in Algorithm~\ref{alg:nhbd}, and inner iterations are used to find the smallest $j$. We remark that the line search implementation is slightly different from Armijo's rule, so it might not perform as well and there is no guarantee on the convergence to the global minimizer. In this sense, this implementation is not optimal and it is just for illustrative purposes.  


\begin{algorithm}[htbp]
\caption{Implementation of Gradient Descend with Line Search} \label{alg:nhbd}
\begin{algorithmic}
\REQUIRE $\bA_1, \bA_2, \cdots, \bA_n \subseteq
\SS_+(p)$: $\nu$: start step size, $c$: control parameter size
\ENSURE $\bX$: the Karcher mean.\\
\textbf{Steps}:
\STATE
     $\bullet$ $\bX_1=\frac{1}{n}\sum_{i=1}^n\bA_i$, $k=1$
    \REPEAT \STATE
     $\bullet$ Let $\bD=\frac{1}{n}\sum_{i=1}^n\log(\bX_k^{-\frac{1}{2}}\bA_i\bX_k^{-\frac{1}{2}})$\\
      $\bullet$ Find the smallest $j>0$ such that $F(\bX_k^{\frac{1}{2}}\exp(c^j\nu \bD)\bX_k^{\frac{1}{2}})<F(\bX_k)$, and let $\bX_{k+1}=\bX_k^{\frac{1}{2}}\exp(c^j\nu \bD)\bX_k^{\frac{1}{2}}$\\
      $\bullet$ $k=k+1$
     \UNTIL Convergence
\end{algorithmic}
\end{algorithm}

The main computation costs of MM, GD and Toolbox algorithms are presented in Table~\ref{tab:comparison}, which includes all steps that have a computational cost of $O(p^3)$. In this sense, all three methods have a computational cost of $O(p^3)$ per iteration (we compare inner iterations of GD algorithm with the iterations in MM and Toolbox algorithms). While it is generally difficult to compare the empirical computational cost without looking into the implementation, we highlight all the matrix functions that require an iterative procedure with each iteration costs $O(p^3)$, since they are more computational expensive than other steps in Table~\ref{tab:comparison}. In our implementations, these computational expensive steps are usually calculated by eigenvalue decomposition with \eqref{eq:matrixfunction}, though there might exist faster implementations, especially for the inner iteration of GD, where only the eigenvalues of $\{\exp(c^j\nu\bD)^{-\frac{1}{2}}\bA_i\exp(c^j\nu\bD)^{-\frac{1}{2}}\}_{i=1}^n$ are needed (that being said, finding eigenvalues still requires an iterative procedure and it is more expensive than matrix multiplication).

Following this implementation, all three algorithms have similar empirical computational complexities per iteration. Their computational costs are mostly from eigenvalue decompositions for the highlighted steps in Table~\ref{tab:comparison}. GD algorithm has $n+1$ such steps per inner iteration and $n+1$ such steps per out iteration; MM algorithm has $n+2$ such steps per iteration (note that the calculation of $g_1$ and $g_2$ can share one eigenvalue decomposition); Toolbox algorithm has $n$ such steps per iteration.

However, the MM algorithm requires more matrix multiplication steps per iteration, compared to GD and Toolbox algorithms. Therefore, the total computational cost depends on the ratio between the computational cost of matrix multiplication and the computational cost of matrix functions highlighted in Table~\ref{tab:comparison}. In our configuration (MATLAB R2014a, Windows 10 64 bits, i5-6300U), the matrix multiplication between two $100\times 100$ matrices takes about $0.13$ milliseconds, finding the eigenvalues of a $100\times 100$ matrix takes about $0.55$ milliseconds, and finding both the eigenvalues and the eigenvectors of a $100\times 100$ matrix takes about $0.96$ milliseconds.

\begin{table*}[htbp]
\centering \caption{\small {Comparison of computational costs in terms of the number of eigenvalue decompositions.}}\label{tab:comparison}
\begin{center}
\begin{tabular}{ l | l }
  \hline
  Algorithm& major computation steps\\
  \hline
  MM & \textbf{$\mathbf{g_1}$ and $\mathbf{g_2}$} of $\{\bA_i^{-\frac{1}{2}}\bX_k\bA_i^{-\frac{1}{2}}\}_{i=1}^n$\\
  MM& \textbf{square root} of $f_2(\bX_k)$ \\
  MM & \textbf{inverse square root} of $f_2(\bX_k)^{\frac{1}{2}}f_1(\bX_k)f_2(\bX_k)^{\frac{1}{2}}$ \\
    MM& additional $5n+4$ matrix multiplications\\
    \hline
 GD, outer iteration & \textbf{square root / inverse square root }of $\bX_k$\\
 GD, outer iteration & \textbf{matrix logarithm} of $\{\bX_k^{-\frac{1}{2}}\bA_i\bX_k^{-\frac{1}{2}}\}_{i=1}^n$\\
 GD, outer iteration & additional $2n$ matrix multiplications\\
 \hline
   \multirow{2}{*}{GD, inner iteration} & \textbf{matrix exponential} of $c^j\nu\bD$, and \\
 & \textbf{inverse square root} of $\exp(c^j\nu\bD)$\\
        GD, inner iteration & \textbf{find eigenvalues} of $\{\exp(c^j\nu\bD)^{-\frac{1}{2}}\bA_i\exp(c^j\nu\bD)^{-\frac{1}{2}}\}_{i=1}^n$\\
        GD, inner iteration & additional $2n$ matrix multiplications\\ 
        \hline
             Toolbox & Cholesky decomposition of $\bX_{k-1}$\\
        Toolbox & matrix inversion of $\bR_k$\\
        Toolbox & \textbf{matrix logarithm }of $\{\bR_k^{-1\,T}\bA_i^{-1}\bR_k^{-1}\}_{i=1}^n$\\
             Toolbox & additional $2n+2$ matrix multiplications\\
\hline
\end{tabular}
\end{center}
\end{table*}

Therefore, each inner iteration of GD algorithm has a similar computational complexity as an iteration of MM or Toolbox. For a fair comparison, the number of inner iterations of the GD algorithm is used in the following simulations.

For simulations, we generate the data set $\bA_1, \bA_2, \cdots, \bA_n$ by the following scheme: $\bA_i=\bU_i\bS_i\bU_i^T$, where $\bU_i$ are random orthogonal matrices (generated by MATLAB command ``orth(rand(p,p))''), and $\bS_i$ are diagonal matrices with entries sampled differently for different simulations. All algorithms are initialized with the arithmetic mean $\frac{1}{n}(\bA_1+\bA_2+\cdots+\bA_n)$. The parameters $\nu$ and $c$ in the GD algorithm are set to be $c=\frac{1}{2}$ and $\nu=\frac{1}{5},\frac{1}{3},1,3,5$.

For the first simulation, the diagonal entries of $\bS_i$ are sampled from a uniform distribution in $[1,10]$, so that the condition number of $\bA_i$ is smaller than $10$. We run the simulations with two  settings $p=n=10$ and $p=n=40$, and the mean error of each iteration over $100$ runs, defined by \begin{equation}\|\sum_{i=1}^n\log(\bX_k^{-\frac{1}{2}}\bA_i\bX_k^{-\frac{1}{2}})\|_F,
\label{eq:precision}\end{equation} is visualized in Figure~\ref{fig:compare1}. We remark that the ideal measure would be $\|\bX_k-\hat{\bX}\|_F$ or $F(\bX_k)-F(\hat{\bX})$, where $\hat{\bX}$ is the global minimizer. However, we do not know the exact $\hat{\bX}$, and this gradient-based measure is used as an alternative (similar measure is used in~\cite[Figure 4.6(c)]{JeurVV2012}).

Figure~\ref{fig:compare1} shows that the convergence rate of the MM algorithm is similar to Toolbox, and slower (but still comparable) than the GD algorithm with the best choice of parameter, i.e., when $\nu=1$. However, the precision of the GD algorithm is not as good as MM or Toolbox, and we remark that similar accuracy is also observed for the ``steepest descent'' implementation in \cite{JeurVV2012}.

To investigate the performance of these algorithms further, an instance of the simulation for $p=n=10$ is recorded in Table~\ref{tab:comparison1}. Some rows in the ``GD algorithm'' column are left empty when more than one inner iterations are used to find the step size $j$, for example, it is shown that for $\nu=3$, usually $2$ inner iterations are needed to find $j$. From this table we can see that the choice of the step size $\nu$ is important for GD algorithm: if it is too small, then the convergence is slow; if it is too large then more than one inner iterations are needed to choose the step size, which also makes the algorithm slower. In the examples in Figure~\ref{fig:compare1}, $\nu=1$ is a good choice. However, $\nu=1$ might not be the best choice for all data sets, which will be exemplified in the next simulation.

 \begin{figure}
 \begin{center}
 \includegraphics[width=0.45\textwidth]{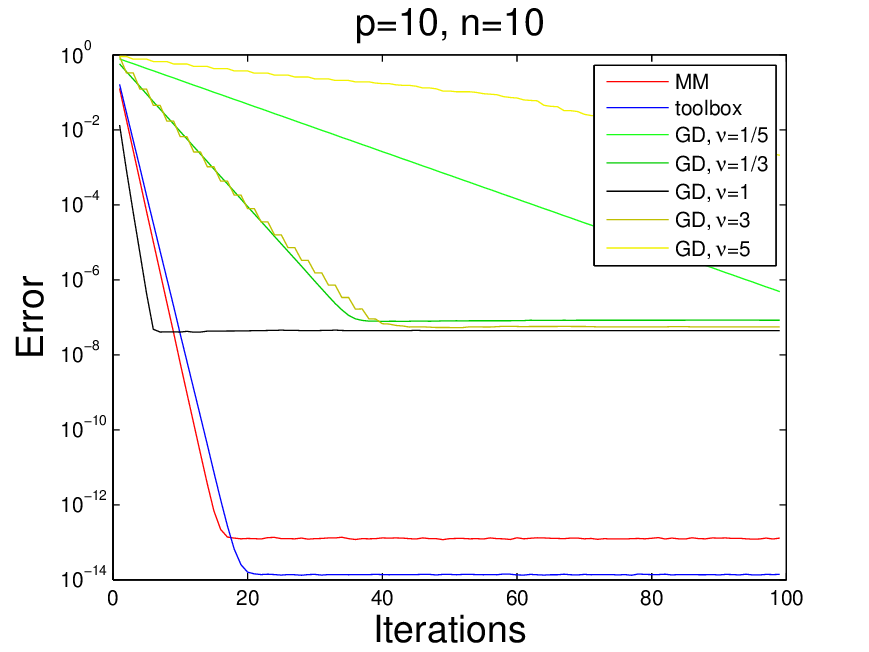}
 \includegraphics[width=0.45\textwidth]{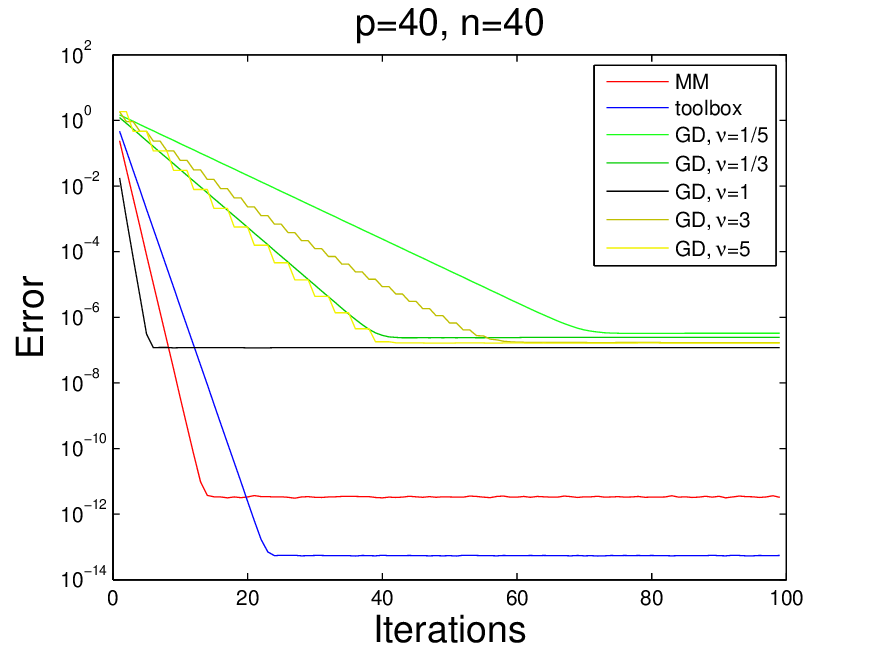}
 \caption{\it The performance of algorithms, where the $x$-axis and the $y$-axis correspond to the number of iterations and error measured by \eqref{eq:precision} respectively.\label{fig:compare1}}
 \end{center}
 \end{figure}

\begin{table*}[htbp]
\centering \caption{\small {The logarithmic errors (with base $10$) of MM, GD and Toolbox algorithms in the first $20$ iterations.}}\label{tab:comparison1}
\begin{center}
\begin{tabular}{ |l| l | l | l | l | l |l |l | }
  \hline
\multirow{2}{*}{iterations}&  \multirow{2}{*}{MM} &\!\multirow{2}{*}{Toolbox}\! &\multicolumn{5}{c|}{GD}\\
&&& \multicolumn{1}{c}{$\nu\!=\!1/5$}& \multicolumn{1}{c}{$\nu\!=\!1/3$}& \multicolumn{1}{c}{$\nu\!=\!1$}& \multicolumn{1}{c}{$\nu\!=\!3$}& \multicolumn{1}{c|}{$\nu\!=\!5$}\\\hline
1&-1.05&-0.94&-0.2&-0.34&-1.98&-0.14&-0.14\\
2&-1.93&-1.74&-0.26&-0.54&-3.27&-0.57&\\
3&-2.8&-2.54&-0.33&-0.74&-4.5&&-0.2\\
4&-3.66&-3.33&-0.39&-0.94&-5.71&-1.01&\\
5&-4.52&-4.13&-0.45&-1.14&-6.91&&\\
6&-5.38&-4.92&-0.52&-1.34&-7.91&-1.44&-0.27\\
7&-6.24&-5.72&-0.58&-1.54& & &\\
8&-7.09&-6.51&-0.65&-1.74& &-1.86&\\
9&-7.95&-7.3&-0.71&-1.94& & &-0.34\\
10&-8.81&-8.1&-0.77&-2.14& &-2.28&\\
11&-9.67&-8.89&-0.84&-2.34& &&\\
12&-10.52&-9.68&-0.9&-2.54& &-2.7&-0.4\\
\hline
\end{tabular}
\end{center}
\end{table*}


In the next simulation, the goal is to find out the performance of the algorithms for matrices with large condition numbers. We let $p=n=10$ and the diagonal entries of $\bS_i$ be a geometric series $10^0, 10^{a}, 10^{2a},\cdots, 10^{9a}$. The results of GD, MM and Toolbox algorithms for the settings $a=0.3, 0.5, 0.7, 0.9$ are visualized in Figure~\ref{fig:compare2}. There are two main observations from this simulation. First, there is no consistent choice of $\nu$ that makes GD perform well. In comparison, MM algorithm and Toolbox are parameter-free and always converge in a reasonable rate. Second, While the convergence rates of all algorithms suffer from the large condition numbers, MM algorithm converges faster than the Toolbox algorithm.

 \begin{figure}
 \begin{center}
 \includegraphics[width=0.45\textwidth]{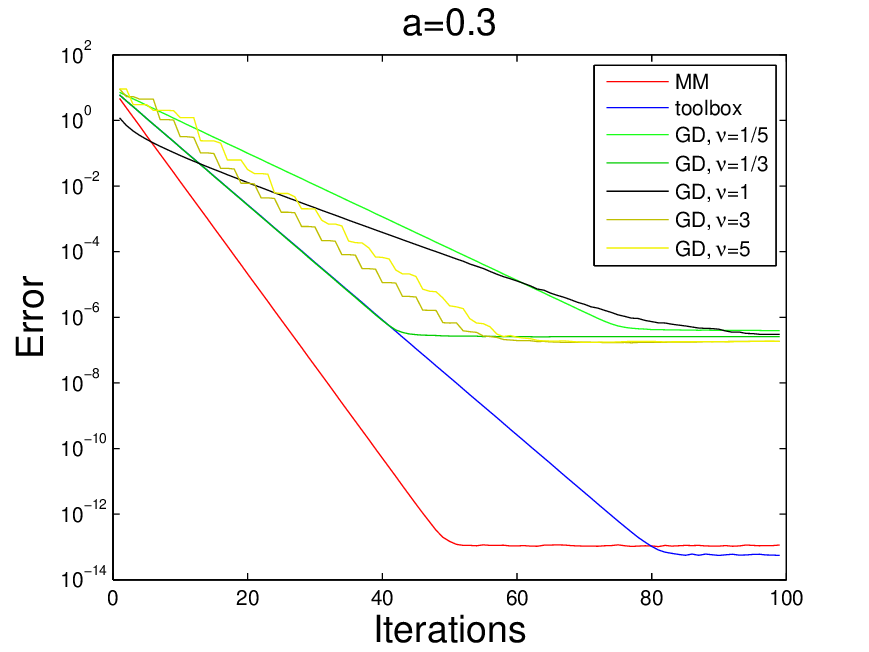}
 \includegraphics[width=0.45\textwidth]{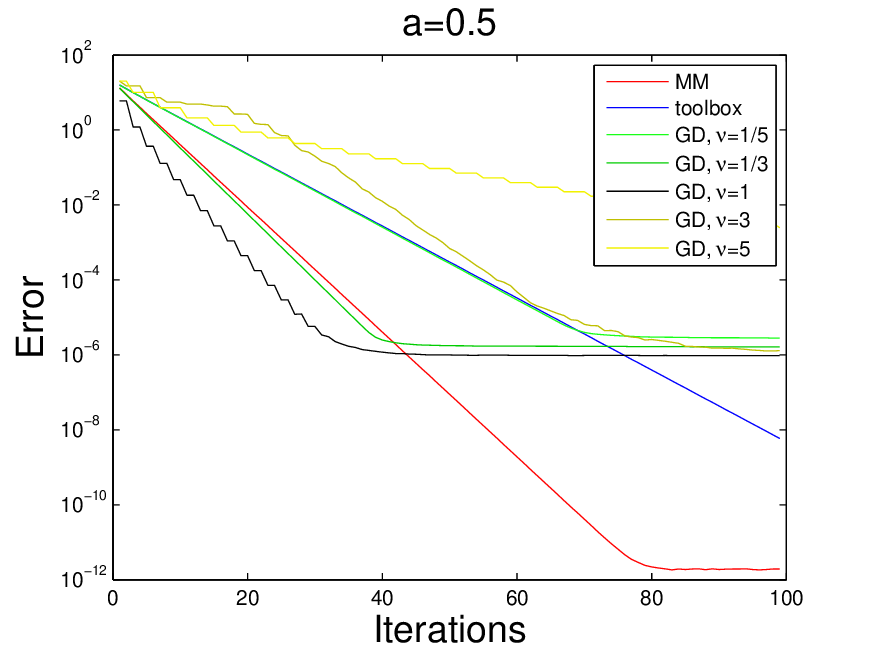}
  \includegraphics[width=0.45\textwidth]{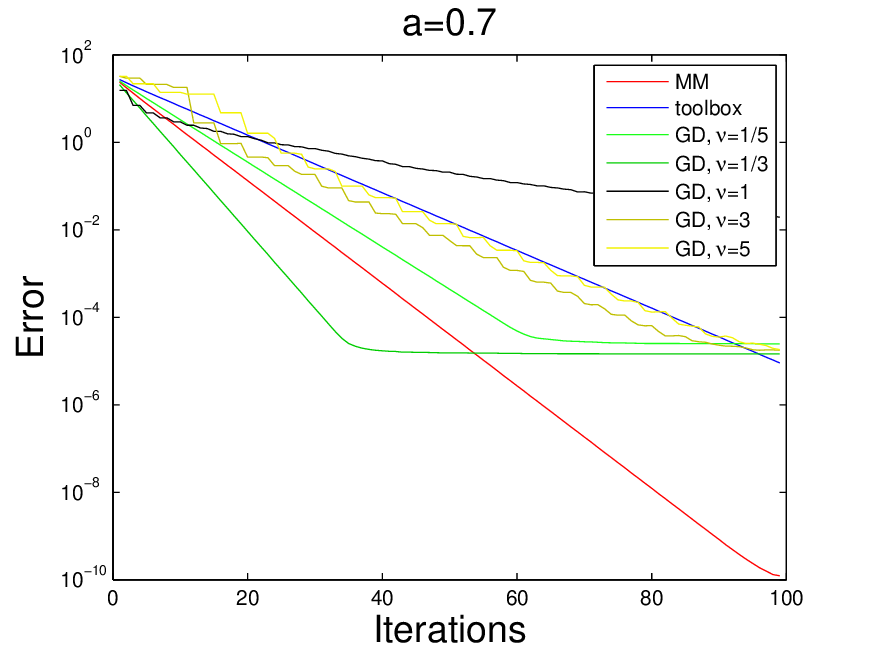}
 \includegraphics[width=0.45\textwidth]{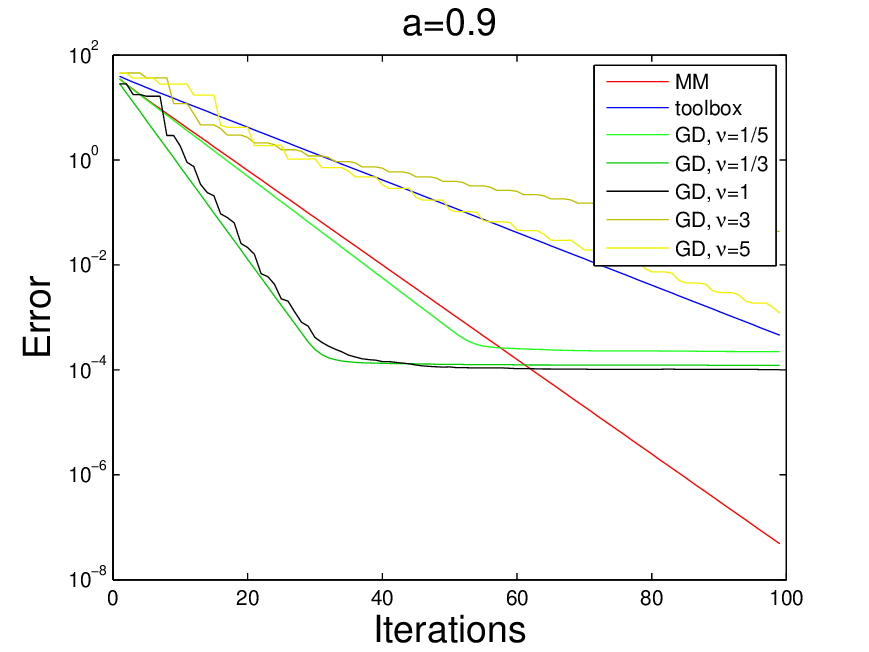}
 \caption{\it The performance of algorithms for a set of matrices with large condition numbers.\label{fig:compare2}}
 \end{center}
 \end{figure}

In simulations we also observed that the convergence rate of MM algorithm is slower when the matrices $\bA_i$ have different scalings, i.e, when one of them is much larger than the other. 
We use the simulation in the left figure of  Figure~\ref{fig:compare1}, and multiply $\bA_1$ by $10^4$. The performance of various algorithms is visualized in the left figure of Figure~\ref{fig:compare3}, which shows that MM algorithm has a slower convergence in the second case. 

We also repeat the simulation in the left figure of  Figure~\ref{fig:compare1}, with $n=200$ (instead of $n=10$) and $p=10$, and record the performance of various algorithm in the right figure of Figure~\ref{fig:compare3}. It shows that MM algorithm is capable of handling a larger number of $n$ without sacrificing much accuracy or convergence rate. However, similar to Figure~\ref{fig:compare1}, its accuracy is not as good as the Toolbox algorithm.
 \begin{figure}
 \begin{center}
 \includegraphics[width=0.45\textwidth]{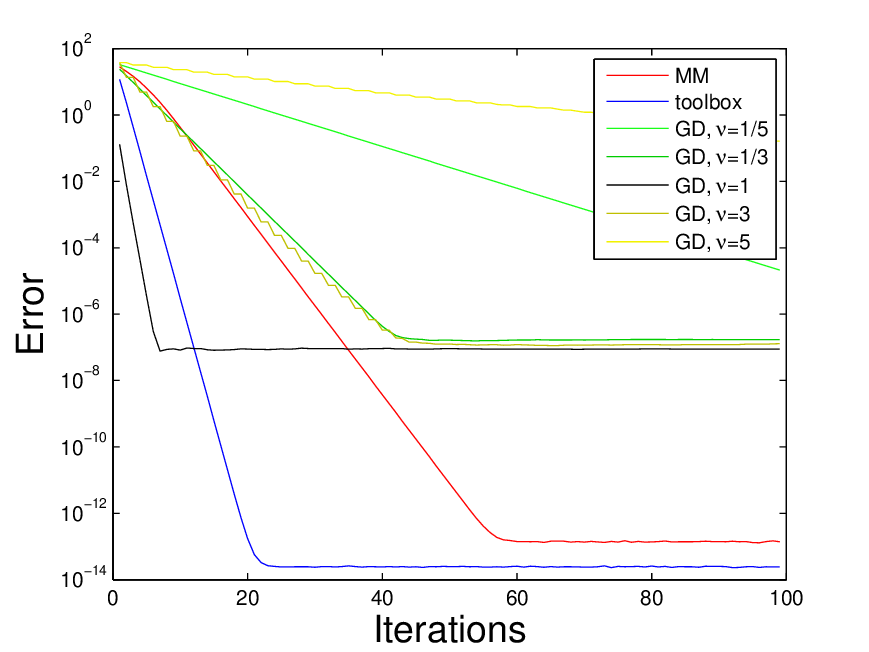}
 \includegraphics[width=0.45\textwidth]{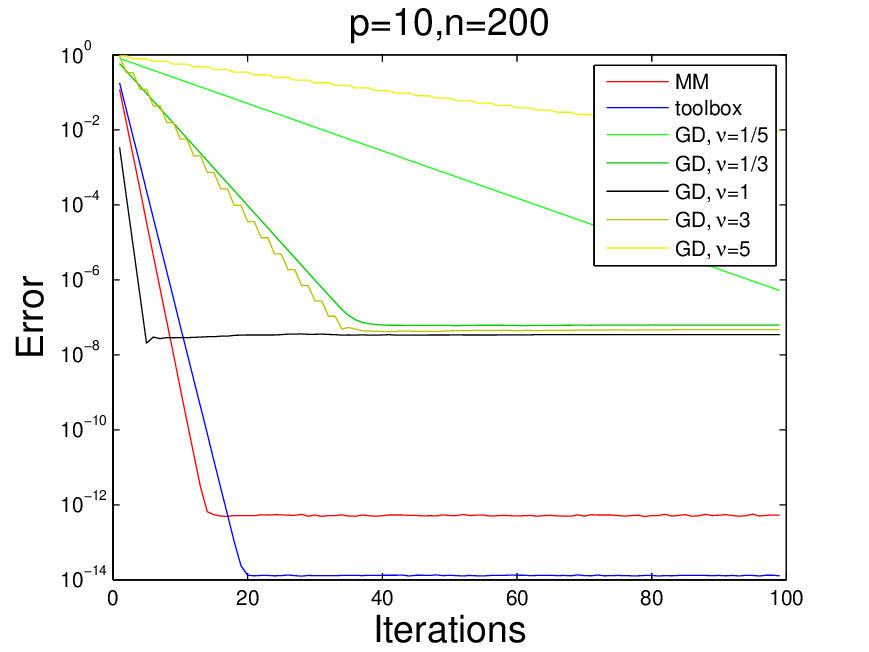}
 \caption{\it Left: the performance of algorithms for matrices with different scalings. Right: the performance of algorithms for $n=200$ and $p=10$.\label{fig:compare3}}
 \end{center}
 \end{figure}

\section{Conclusion}
This paper has presented a novel algorithm for computing the Karcher mean of positive definite matrices based on the majorization-minimization (MM) principle. The MM algorithm is simple to implement and has a theoretical convergence guarantee. Compared with the standard gradient descent algorithm, this algorithm does not need to choose a step size in each iteration. Compared with the linearized gradient descent algorithm in~\cite{Bini2013}, it has a global convergence guarantee and from the experiments considered in the paper, it converges faster when the condition numbers of the matrices are large. However, the accuracy of the MM algorithm is not as good, which might be due to the implementation.

There are some possible future directions arising from this work. First, the MM algorithm strongly depends on the choice of the majorization function (in our case, the function $G(\bX,\bX')$), and it would be interesting to investigate that if other majorization functions give better performance. Second, it would also be interesting to apply the framework of MM algorithms to other manifold optimization problems.

\bibliographystyle{abbrv}
\bibliography{bib-online}
\end{document}